\documentclass[10pt]{amsart}
\usepackage[english]{babel}
\usepackage{amssymb}
\usepackage{amsmath}
\usepackage{srcltx}
\usepackage{lscape}

\newtheorem{dummy}{Dummy}

\newtheorem{lemma}[dummy]{Lemma}
\newtheorem{theorem}[dummy]{Theorem}
\newtheorem{proposition}[dummy]{Proposition}
\newtheorem{corollary}[dummy]{Corollary}

\theoremstyle{definition}

\newtheorem{example}[dummy]{Example}

\newtheorem{remark}[dummy]{Remark}

\newcommand{\ignore}[1]{}

\author{C. Brown}
\author{S. Pumpl\"un}
\author{A. Steele}

\email{Christian.Brown@nottingham.ac.uk; susanne.pumpluen@nottingham.ac.uk;
andrew.steele@aquaq.co.uk}
\address{School of Mathematical Sciences\\
University of Nottingham\\ University Park\\ Nottingham NG7 2RD\\
United Kingdom }



\begin{document}

\title[Automorphisms and isomorphisms of semifields]
{Automorphisms and isomorphisms of Jha-Johnson semifields obtained from skew polynomial rings}

\begin{abstract}
We study the automorphisms of Jha-Johnson semifields  obtained from a right invariant irreducible twisted polynomial
$f\in K[t;\sigma]$, where $K=\mathbb{F}_{q^n}$ is a finite field and $\sigma$ an automorphism of $K$ of order $n$,
with a particular emphasis on inner automorphisms and the automorphisms of  Sandler and Hughes-Kleinfeld  semifields.
We include the automorphisms of some Knuth semifields (which do not arise from skew polynomial rings).

Isomorphism between Jha-Johnson semifields are considered as well.
\end{abstract}

\maketitle

\section*{Introduction}

  Semifields are finite unital nonassociative division algebras. Since
two semifields coordinatize the same non-Desarguesian projective
plane if and only if they are isotopic, semifields are usually classified up to isotopy rather than up to isomorphism
 and consequently, usually only their autotopism group is computed.

 Among the semifields with known automorphism groups
  are the  three-dimensional semifields over a field of characteristic not $2$
 (Dickson \cite{D06} and Menichetti \cite{M73,M77}), and the semifields with 16 elements
(Kleinfeld \cite{kleinfeld1960techniques} and Knuth \cite{Kn65}). Burmester
\cite{burmester1962commutative} investigated the automorphisms of Dickson commutative
semifields of order $p^{2n}, p \neq 2,$ and
Zemmer
\cite{zemmer1952subalgebras} proved the existence of commutative
semifields with a cyclic automorphism group of order $2n$.
More recent results can be found for instance in \cite{A1,A2,A3,A4,A5,A6,B1,B2}.

Jha-Johnson semifields (also called \emph{cyclic semifields} in  \cite{Dem}) were introduced in \cite{JJ}
and are generalizations of both Sandler and Hugh-Kleinfeld semifields.
With the exception of one subcase, the autotopism groups of all Jha-Johnson semifields were computed by Dempwolff
using representation theory \cite{Dem}.
One of our motivations for computing the automorphism groups of a certain family of Jha-Johnson  semifields
 is a question by C. H. Hering \cite{H}: Given a finite group $G$, does there exist a
semifield  such that $G$ is a subgroup of its automorphism group?
Another incentive arose from the search for classes of finite loops with non-trivial automorphism groups.
Examples can be now obtained as the multiplicative loops
of Jha-Johnson semifields \cite{P17}.

 We will compute the automorphism groups using noncommutative polynomials:
Let $K$ be a field, $\sigma$  an automorphism of $K$ with fixed field
$F$,  $R = K[t;\sigma]$ a twisted polynomial ring and $f\in R$. In 1967, Petit \cite{P66,P68}  studied
 a class of unital nonassociative algebras $S_f$ obtained by employing a right invariant irreducible  $f\in R = K[t;\sigma]$.

Every finite nonassociative Petit algebra is a \emph{Jha-Johnson semifield}.
 These algebras were studied by Wene \cite{W00} and more recently, Lavrauw and Sheekey \cite{LS}.

 While each Jha-Johnson semifield is isotopic to some
such algebra $S_f$ it is not necessarily itself isomorphic to an algebra  $S_f$. We will focus on
those Jha-Johnson semifields which are, and apply the results from \cite{BP} to investigate their automorphisms.

The structure of the paper is as follows: In Section \ref{sec:prel},
we introduce the basic terminology and define the algebras $S_f$.
 Given  a finite field $K=\mathbb{F}_{q^n}$,  an automorphism $\sigma$ of $K$ of order $n$ with $F={\rm Fix}(\sigma)=
 \mathbb{F}_q$ and  an irreducible polynomial $f\in K[t;\sigma]$ of degree $m$ that is not right
 invariant (i.e., where $ K[t;\sigma]f$ is not a two-sided ideal), we know
 the automorphisms of the Jha-Johnson semifields $S_f$
if  $n\geq m-1$ and a subgroup of them if $n<m-1$  \cite[Theorems 4, 5]{BP}.
 The automorphism groups of \emph{Sandler semifields} \cite{San62} (obtained by choosing $n\geq m$ and
 $f(t) = t^m - a\in K[t; \sigma]$,  $a \in K \setminus F$) are particularly relevant: for
all Jha-Johnson semifields $S_g$ with $g(t) = t^m - \sum_{i=0}^{m-1} b_i t^i
\in K[t;\sigma]$ and $b_0= a$, ${\rm Aut}_F(S_g)$ is a subgroup of ${\rm Aut}_F(S_f)$ (Theorem \ref{Aut(S_f)
subgroup corollary}). We summarize results on the automorphism groups,
and give examples when it is trivial and when ${\rm Aut}_F(S_f)\cong  \mathbb{Z}/n\mathbb{Z}$
(Theorem \ref{thm:automorphism_of_Sf_field_caseIV}).
 Inner automorphisms of Jha-Johnson semifields are considered in Section \ref{sec:inner}.
In Section \ref{sec:nonasscyclicfinite} we consider the special case that $n=m$ and
$f(t)=t^m-a$. In this case, the algebras $S_f$ are examples of Sandler semifields and also called
\emph{nonassociative cyclic algebras} $(K/F,\sigma,a)$. The automorphisms of $A=(K/F,\sigma,a)$  extending $id$ are
inner and form a cyclic group
isomorphic to ${\rm ker}(N_{K/F})$. We show when ${\rm Aut}_F(A)\cong {\rm ker}(N_{K/F})$
and hence consists only of inner automorphisms, when
${\rm Aut}_F(A)$ contains or equals  the dicyclic group ${\rm Dic}_r$ of order $4r=2q+2$, or when
${\rm Aut}_F(A) \cong \mathbb{Z} / ( s/m )\mathbb{Z} \rtimes_q \mathbb{Z} / (m^2) \mathbb{Z}$
contains or equals a semidirect product,
where $s = (q^m-1)/(q-1)$, $m>2$ (Theorems \ref{thm:Automorphisms of nonassociative cyclic algebras over finite fields} and \ref{thm:semidirect}).
 We  compute the automorphisms for the Hughes-Kleinfeld
 and most of the Knuth semifields in Section \ref{sec:knuth}. Not all Knuth semifields are algebras $S_f$, however, the
 automorphisms behave similarly in all but one case. We compute the automorphism groups in some examples,
 improving results obtained by Wene \cite{W09}.
 In Section \ref{sec:isofinite} we briefly investigate the isomorphisms between two semifields $S_f$ and $S_g$.
In particular, we classify  nonassociative cyclic algebras of prime
degree up to isomorphism.

 Sections of this work are part of the first and last  author's PhD theses
 \cite{CB, AndrewPhD} written under the supervision of
 the second author.

%
%

\section{Preliminaries} \label{sec:prel}


\subsection{Nonassociative algebras} \label{subsec:nonassalgs}


Let $F$ be a field and let $A$ be an $F$-vector space. $A$ is an
\emph{algebra} over $F$ if there exists an $F$-bilinear map $A\times
A\to A$, $(x,y) \mapsto x \cdot y$, denoted by juxtaposition
$xy$, the  \emph{multiplication} of $A$. An algebra $A$ is called
\emph{unital} if there is an element in $A$, denoted by 1, such that
$1x=x1=x$ for all $x\in A$. We will only consider unital algebras without saying so explicitly.

The {\it associator} of $A$ is given by $[x, y, z] =
(xy) z - x (yz)$. The {\it left nucleus} of $A$ is defined as $\mathbb{N}_l(A) = \{ x \in A \, \vert \, [x, A, A]  = 0 \}$, the {\it
middle nucleus} of $A$ is $\mathbb{N}_m(A) = \{ x \in A \, \vert \,
[A, x, A]  = 0 \}$ and  the {\it right nucleus} of $A$ is
$\mathbb{N}_r(A) = \{ x \in A \, \vert \, [A,A, x]  = 0 \}$. $\mathbb{N}_l(A)$, $\mathbb{N}_m(A)$, and $\mathbb{N}_r(A)$ are associative
subalgebras of $A$. Their intersection
 $\mathbb{N}(A) = \{ x \in A \, \vert \, [x, A, A] = [A, x, A] = [A,A, x] = 0 \}$ is the {\it nucleus} of $A$.
$\mathbb{N}(A)$ is an associative subalgebra of $A$ containing $F1$
and $x(yz) = (xy) z$ whenever one of the elements $x, y, z$ lies in
$\mathbb{N}(A)$.   The
 {\it center} of $A$ is ${\rm C}(A)=\{x\in A\,|\, x\in \text{Nuc}(A) \text{ and }xy=yx \text{ for all }y\in A\}$.

An algebra $A\not=0$ is called a \emph{division algebra} if for any
$a\in A$, $a\not=0$, the left multiplication  with $a$, $L_a(x)=ax$,
and the right multiplication with $a$, $R_a(x)=xa$, are bijective.
 If $A$ has finite dimension over $F$, $A$ is a division algebra if
and only if $A$ has no zero divisors \cite[ pp. 15, 16]{Sch}.
 A \emph{semifield} is a finite-dimensional unital division algebra over a finite field.
 A semifield is called \emph{proper} if it is not associative.
An element $0\not=a\in A$ has a \emph{left inverse} $a_l\in A$, if
$R_a(a_l)=a_l a=1$, and a \emph{right inverse}
 $a_r\in A$, if $L_a(a_r)=a a_r=1$. If  $m_r=m_l$ then we denote this element by
$m^{-1}$.

An automorphism $G\in {\rm Aut}_F(A)$ is an \emph{inner automorphism}
if there is an element $m\in A$ with left inverse $m_l$ such
that $G(x) =G_m(x)= (m_lx)m$ for all $x\in A$.
The set of inner automorphisms $\{G_m\,|\, m\in \mathbb{N}(A) \text{ invertible} \}$ is a subgroup of ${\rm Aut}_F(A)$.
Note that if the nucleus of $A$ is commutative, then for all $0\not=n\in
\mathbb{N}(A)$, $G_n(x)=(n^{-1}x)n=n^{-1}x n$ is an inner automorphism of $A$
such that ${G_n}|_{\mathbb{N}(A)}=id_{\mathbb{N}(A)}$.


\subsection{Semifields obtained from skew polynomial rings} \label{subsec:structure}


Let $K$ be a field and $\sigma$ an automorphism of $K$. The \emph{twisted polynomial ring} $R=K[t;\sigma]$
is the set of polynomials $a_0+a_1t+\dots +a_nt^n$ with $a_i\in K$,
where addition is defined term-wise and multiplication by
$ta=\sigma(a)t $ for all $a\in K$ \cite{O1}.
For $f=a_0+a_1t+\dots +a_mt^m$ with $a_m\not=0$ define ${\rm
deg}(f)=m$ and put ${\rm deg}(0)=-\infty$. Then ${\rm deg}(fg)={\rm deg}
(f)+{\rm deg}(g).$
 An element $f\in R$ is \emph{irreducible} in $R$ if it is not a unit and  it has no proper factors,
 i.e if there do not exist $g,h\in R$ with
 ${\rm deg}(g),{\rm deg} (h)<{\rm deg}(f)$ such
 that $f=gh$.

 $R=K[t;\sigma]$ is a left and right principal ideal domain
 and there is a right division algorithm in $R$: for all
$g,f\in R$, $g\not=0$, there exist unique $r,q\in R$ with ${\rm
deg}(r)<{\rm deg}(f)$, such that $g=qf+r$
\cite{J96}.
From now on, we assume that
$$K = \mathbb{F}_{q^n}$$
is a finite field,  $q = p^r$ for some prime $p$, $\sigma$ an automorphism of $K$ of order $n>1$ and
$$F={\rm Fix}(\sigma)=\mathbb{F}_{q},$$
 i.e. $K/F$ is a cyclic Galois extension of degree $n$
with ${\rm Gal}(K/F)=\langle\sigma\rangle.$  The norm  $N_{K / F} : K^{\times} \rightarrow F^{\times}$ is surjective, and
${\rm ker}(N_{K/F})$ is a cyclic group of order
$s =(q^n-1)/(q-1).$

Let $f\in R=K[t;\sigma]$ have degree $m$. Let
${\rm mod}_r f$ denote the remainder of right division by $f$. Then the additive abelian group
$R_m=\{g\in K[t;\sigma]\,|\, {\rm deg}(g)<m\}$
 together with the multiplication
 $g\circ h=gh \,\,{\rm mod}_r f $
 is a unital nonassociative algebra $S_f=(R_m,\circ)$ over $F$ \cite[(7)]{P66}.
  $S_f$ is also denoted by $R/Rf$
 if we want to make clear which ring $R$ is involved in the construction.

 Note that using left division by $f$ and
 the remainder ${\rm mod}_l f$ of left division by $f$ instead, we can analogously define the multiplication for
  another unital nonassociative algebra on $R_m$ over $F_0$, called $\,_fS$. Every algebra $\,_fS$ is the opposite algebra of some Petit algebra \cite[(1)]{P66}.

  In the following, we call the algebras $S_f$  \emph{Petit algebras} and denote their multiplication  by juxtaposition.
Without loss of generality, we will only  consider monic $f(t)$, since $S_f= S_{af}$ for all  $a\in K^\times$.
Let
$$f(t) = t^m -\sum_{i=0}^{m-1} a_i t^i \in K[t;\sigma].$$
 $S_f$ is a semifield if and only if $f$ is irreducible, and a proper semifield
 if and only if $f$ is not right invariant (i.e., the left ideal $Rf$ generated by $f$ is not two-sided),
 cf. \cite[(2), p.~13-03, (5), (9)]{P66}, or \cite{LS}.
 If $S_f$ is a proper semifield then $\mathbb{N}_l(S_f)=\mathbb{N}_m(S_f)=K\cdot 1= \mathbb{F}_{q^n}\cdot 1$ and
 $$\mathbb{N}_r(S_f)=\{g\in R\,|\, fg\in Rf\}\cong\mathbb{F}_{q^m}$$
 \cite{LS}. $S_f$ has order $q^{mn}$.
 The powers of $t$ are associative if and only if $t^mt=tt^m$
 if and only if $t\in \mathbb{N}_r(S_f)$ if and only if $ft\in Rf.$

 In particular, let $f(t) \in F[t]=F[t;\sigma]\subset K[t;\sigma]$ be monic, irreducible and not right invariant.
 Then
$$F[t]/(f(t))\cong F\oplus Ft\oplus\dots\oplus Ft^{m-1}\cong \mathbb{N}_r(S_f)$$
and thus $\mathbb{N}(S_f)=F\cdot 1.$
Moreover, we have $ft\in Rf$ \cite[Proposition 3]{BP}.

\begin{remark}
Note that  $f(t)\in K[t;\sigma]\setminus F[t;\sigma]$ is never right
invariant and that if $f(t)\in  F[t]\subset  K[t;\sigma]$ has degree $m<n$, then $f(t)$ is never right invariant, either.
 For $n=m$ the only rght invariant $f(t)\in  F[t]$ are of the form $f(t)=t^m-a$,
 and these polynomials are not irreducible. So for  $n=m$, all irreducible polynomials in $  F[t]$ are not right invariant.
\end{remark}

If the semifield $A=K[t;\sigma]/K[t;\sigma]f$ has a nucleus
which is larger than its center,
then the inner automorphisms $\{G_c\,|\, 0\not=c\in \mathbb{N}(A) \}$ form a non-trivial
subgroup of ${\rm Aut}_F(S_f)$ \cite[Lemma 2, Theorem 3]{W09} and each such inner automorphism
$G_c$ extends $id_{\mathbb{N}(A)}$.

We will assume throughout the paper that $f\in K[t;\sigma]$ is irreducible of degree $ m \geq
2$, since if $f$ has degree $1$ then  $S_f\cong K$, and that $\sigma\not=id$.

We will always choose irreducible polynomials $f\in K[t;\sigma]$  which are not right invariant, which is equivalent to
 $S_f$ being  a proper semifield.
 Each Jha-Johnson semifield is isotopic to some
Petit algebra $S_f$ \cite[Theorem 16]{LS} but not necessarily a Petit algebra itself. We will focus on
those Jha-Johnson semifields which are Petit algebras $S_f$, and apply the results from \cite{BP}.

%
%

\subsection{Automorphisms of Jha-Johnson semifields  $S_f$}\label{sec:semifields}

 Assume that
$$f(t) = t^m -\sum_{i=0}^{m-1} a_i t^i  \in K[t; \sigma]$$
 has degree $m$, is monic, irreducible and not right invariant. Then
 $S_f$ is a Jha-Johnson semifield over $F=\mathbb{F}_q$
\cite{LS}. We recall some results from \cite{BP} for the convenience of the reader:
Let $\tau=\sigma^j\in {\rm Gal}(K/F)$ for some $j$, $0\leq j\leq n-1$ and $k \in K^{\times}$, such that
\begin{equation} \label{eqn:neccessary}
\tau(a_i) = \Big( \prod_{l=i}^{m-1}\sigma^l(k) \Big) a_i
\end{equation}
for all $i \in\{ 0, \ldots, m-1\}$.
Then the map $H:S_f\longrightarrow S_f$,
$$H_{\tau , k}(\sum_{i=0}^{m-1} x_i t^i )= \tau(x_0) + \tau(x_1)kt + \tau(x_2)k\sigma(k)t^2+\cdots +
\tau(x_{m-1}) k \sigma(k) \cdots \sigma^{m-2}(k) t^{m-1}$$
is an automorphism of $S_f$. These $H_{\tau , k}$  form a subgroup of  ${\rm Aut}_F(S_f)$.
In particular, if  $n\geq m-1$ then
$${\rm Aut}_F(S_f)=\{H_{\tau , k}\,|\, \text{ with } \tau=\sigma^j, 0\leq j\leq n-1 \text{ and } k \in K^{\times}
\text{ satisfying Equation }(\ref{eqn:neccessary})\}$$
\cite[Theorems 4, 5]{BP}.

 An algebra $S_f$ with $f(t) = t^m - a\in K[t; \sigma]$,  $a \in K \setminus F$ and $n\geq m$
 is called a \emph{Sandler semifield} \cite{San62}. For $m=n$, these algebras are also called
\emph{nonassociative cyclic (division) algebras of degree $m$}, as they can be seen as canonical
generalizations of associative cyclic algebras (since for $a\in F^\times$, $S_f$
with $f(t) = t^m - a\in K[t; \sigma]$ is a classical cyclic algebra
of degree $m$ as defined in \cite[p. 414]{KMRT}). These algebras are treated in Section \ref{sec:nonasscyclicfinite}.

 The automorphism groups of Sandler semifields
 are particularly relevant:

\begin{theorem} \cite[Theorem 8]{BP}\label{Aut(S_f) subgroup corollary}
 Let $n\geq m-1$ and
 $g(t) = t^m - \sum_{i=0}^{m-1} b_i t^i \in K[t;\sigma]$
  be irreducible and not right invariant. Assume one of the following:
\\ (i)  $b_0 \in K \setminus F$ and $f(t) = t^m - b_0 \in K[t;\sigma]$.
\\ (ii)  $f(t) = t^m - \sum_{i=0}^{m-1} a_i t^i \in K[t;\sigma]$ be  
such that $b_i \in \left\{ 0 , a_i \right\}$ for all $i \in \{ 0,\ldots , m-1\}$.
\\ Then
${\rm Aut}_F(S_g)\subset {\rm Aut}_F(S_f)$  is a subgroup.
\end{theorem}

\begin{theorem} \cite[Theorem 9]{BP}
 Let $n < m-1$ and $g(t) = t^m - \sum_{i=0}^{m-1} b_i t^i \in K[t; \sigma]$ be irreducible and not
be right invariant. Assume one of the following:
\\ (i)  $f(t) = t^m - b_0 \in K[t; \sigma]$ with $b_0 \in K \setminus F$.
\\ (ii)  $f(t) = t^m - \sum_{i=0}^{m-1} a_i t^i \in K[t;\sigma]$ be 
such that $b_i \in \left\{ 0 , a_i \right\}$ for all $i \in \{ 0,\ldots , m-1\}$.
\\ Then  $$\{H\in{\rm Aut}_F(S_g)\,|\, H=H_{\tau , k}
\}\text{ is a subgroup of }\{H\in {\rm Aut}_F(S_f)\,|\, H=H_{\tau ,
k} \}.$$
\end{theorem}

 \begin{theorem} ( \cite[Theorem 5, Remark 12, Theorem 11]{BP})\label{thm:automorphism_of_Sf_field_caseIV}
Suppose $a_{m-1}\in F^\times$, or that two consecutive coefficients $a_s$ and $a_{s+1}$  lie in $F^\times$.
\\ (i) For  $n\geq m-1$ we distinguish two cases:
\\  If $a_i\not\in {\rm Fix}(\tau)$ for all $\tau\not=id$ and all non-zero
$a_i$, $i\not=m-1$,  then
${\rm Aut}_F(S_f)=\{id\}$.
\\  If $f(t)\in F[t]\subset K[t;\sigma]$ then any automorphism $H$ of $S_f$ has the form $H_{\tau , 1}$
where  $\tau\in {\rm Gal}(K/F)$,
and
$${\rm Aut}_F(S_f)\cong  \mathbb{Z}/n\mathbb{Z}.$$
 (ii) Let  $n < m-1$. If $f(t)\in F[t] \subset K[t;\sigma]$ is not right invariant, then for all $\tau\in {\rm Gal}(K/F)$,
  the maps  $H_{\tau , 1}$  are automorphisms of $S_f$ and
 $ \mathbb{Z}/n\mathbb{Z}$ is isomorphic to a subgroup of  ${\rm Aut}_F(S_f)$.
\end{theorem}

\begin{theorem} (\cite[Theorem 19]{BP})\label{thm:fixedfieldfinite}
Let $f(t)  \in F[t]\subset K[t;\sigma]$.
\\ (i) $\langle H_{\sigma,1}\rangle$ is a cyclic subgroup of ${\rm Aut}_F(S_f)$ of order $n$.
\\ (ii) Suppose one of the following holds:
\\ (a) $n\geq m-1$ and  $a_{m-1} \in F^\times$.
\\ (b) $n=m$ is prime, $a_0 \neq 0$
 and at least one of
$a_1, \ldots, a_{m-1}$ is non-zero.
\\ Then
${\rm Aut}_F(S_f) = \langle H_{\sigma,1}\rangle\cong \mathbb{Z}/n \mathbb{Z}$
 and any automorphism extends exactly one $\tau\in {\rm Gal}(K/F)$.
\end{theorem}

\begin{proposition}  \cite[Corollaries 13, 14]{BP} \label{cor:Sandler}
 Let   $f(t)=t^m-a\in K[t;\sigma]$, $a\in K\setminus F$ and $\tau\in {\rm Gal}(K/F)$.
\\ (i) For all  $k \in K^{\times}$ with
$$ \tau(a) = \Big( \prod_{l=0}^{m-1}\sigma^l(k) \Big) a,$$
 the maps  $H_{\tau , k}$ are automorphisms of $S_f$. In particular,
 $N_{K/F}(k)$ is an $m$th root of unity. If  $n\geq m-1$ these are all automorphisms of $S_f$.
\\ (ii) For all $g(t) = t^m - \sum_{i=0}^{m-1} a_i t^i \in K[t; \sigma]$ with $a_0=a$,
$$\{H\in {\rm Aut}_F(S_g)\,|\, H=H_{\tau , k} \}\text{ is a subgroup
of }\{H\in {\rm Aut}_F(S_f)\,|\, H=H_{\tau , k} \}.$$
If  $n\geq m-1$ then these groups are the automorphism groups of $S_g$ and
$S_f$, hence in that case ${\rm Aut}_F(S_g)$ is a subgroup of ${\rm Aut}_F(S_f)$.
\end{proposition}

\begin{corollary} \label{cor:SandlerII}
Let  $n \geq m-1$ and  $f(t)=t^m-a\in K[t;\sigma]$ with $a\in K\setminus F$. Let $m $ and $(q-1)$ be coprime.
\\ (i) There are at most $s=(q^n-1)/(q-1)$ automorphisms extending each
$\tau=\sigma^j$.
\\ (ii) For all irreducible $g(t) = t^m - \sum_{i=0}^{m-1} a_i t^i \in K[t; \sigma]$ with $a_0=a$, ${\rm Aut}_F(S_g)$ is a subgroup of
${\rm Aut}_F(S_f)$.
\end{corollary}

\begin{proof} (i)
We know that $H\in {\rm Aut}_F(S_f)$ if and only if $H=H_{\sigma^j , k}$
 where  $j \in \left\{ 0, \ldots, n-1 \right\}$ and $k \in K^{\times}$ is such that
 $$ \sigma^j(a) = \Big(\prod_{l=0}^{m-1}\sigma^l(k) \Big) a $$
by Proposition \ref{cor:Sandler}. In particular,
$N_{K/F}(k)=1$. So there are at most $s=(q^n-1)/(q-1)$ automorphisms extending each $\sigma^j$.
\\ (ii) is obvious.
\end{proof}

%
%

\section{Inner automorphisms} \label{sec:inner}

Let $f\in K[t; \sigma]$ have degree $m$, and be monic, irreducible and not right invariant.
 \cite[Corollary 5]{W09} yields immediately:

\begin{proposition} \label{prop:innerII}
Suppose $N=\mathbb{N}(S_f)=\mathbb{F}_{q^l}\cdot 1$ for some integer $1<l\leq n$. Then  $S_f$ has at least
$$(q^l-1)/(q-1)$$
inner automorphisms,  determined by those $q^l$
elements in its nucleus that do not lie in $F$. They all are extensions of $id_N$.
\\ In particular, if $S_f$ has nucleus $K\cdot 1$ then there are
$s=(q^n-1)/(q-1)$
  inner automorphisms of $S_f$ and all
extend $id_K$; thus all have the form $H_{id,k}$ for a suitable $k\in K^\times$.
\end{proposition}

 Every automorphism $H_{id,k}\in {\rm Aut}_F(S_f)$ such that $N_{K/F}(k) = 1$ is an inner automorphism.
If  $$n\geq m-1\text{ and }a_{m-1}\not=0$$ or if
$$n=m,\quad a_i=0 \text{ for all } i\not=0\text{ and } a_0\in K\setminus F$$
 these are all the automorphisms extending $id_K$ \cite[Theorem 16]{BP}.

Let
$\Delta^{\sigma}(l)=\{\sigma(c)lc^{-1}\,|\, c\in K^\times\}$
 denote the $\sigma$-conjugacy class of $l$ \cite{LL}.
By Hilbert's Theorem 90,
 ${\rm ker}(N_{K/F})=\Delta^{\sigma}(1).$
 In particular, for every $a\in F^{\times}$ there exist exactly
$s=(q^n-1)/(q-1)$ elements $u\in K$ with $N_{K/F}(u) = a$.

\begin{proposition}\label{prop:estimate}
Let  $n\geq m-1$.
Then there exist at most
$$|{\rm ker}(N_{K/F})|=(q^n-1)/(q-1)$$
 distinct automorphisms of $S_f$ of the form
$H_{id,k}$ such that $N_{K/F}(k) = 1$. These are inner.
\end{proposition}

\begin{proof}
Every automorphism $H_{id,k}\in {\rm Aut}_F(S_f)$ extending $id_K$ such that $N_{K/F}(k) = 1$ is an inner automorphism
by \cite[Theorem 16]{BP}.
 More precisely, for any $k,l\in K^\times$ with $N_{K/F}(k) = 1=N_{K/F}(l)$  there are $c,d\in  K^{\times}$ such that
 $k=c^{-1}\sigma(c)$, $l=d^{-1} \sigma(d)$, and $H_{id,k} =G_c$, $ H_{id,l}=G_d$ (cf. the proof of \cite[Theorem 16]{BP}). We have
  $$H_{id,k} = H_{id,l} \text{ if and only if }
c^{-1} \sigma(c) = d^{-1} \sigma(d).$$ Therefore there exist at
most $|{\rm ker}(N_{K/F})|=|\Delta^{\sigma}(1)|$ distinct automorphisms of $S_f$ of the
form $H_{id,k}$.
\end{proof}

Proposition \ref{prop:estimate} and Proposition
\ref{prop:innerII} imply the following estimates for the number of inner automorphisms of $S_f$:

\begin{theorem}\label{cor:innerII}
Let   $n \geq m-1$.
 If $S_f$ has nucleus  $K\cdot 1$ then it has $s=(q^n-1)/(q-1)$ inner automorphisms extending $id_K$.
 These form a cyclic subgroup of ${\rm Aut}_F(S_f)$ isomorphic to ${\rm ker}(N_{K/F})$.
\end{theorem}

\begin{proof}
By Proposition \ref{prop:estimate}, there are at most $|{\rm ker}(N_{K/F})|=(q^n-1)/(q-1)$ distinct automorphisms
$H_{id,k}$ of $S_f$ and all of these are inner and extend $id_K$.
By Proposition
\ref{prop:innerII}, if $S_f$ has nucleus $K\cdot 1=\mathbb{F}_{q^n}\cdot 1$ then there exist at least
$s=(q^n-1)/(q-1)$ inner automorphisms, all extending $id_{K}$,  those
determined by the elements in its nucleus which do not lie in $F$. Then there are exactly $s$ inner automorphisms.
\end{proof}

Thus if  $N=\mathbb{N}(S_f)=\mathbb{F}_{q^l}\cdot 1$, $l>1$, is strictly contained in $K\cdot 1$, then
  $S_f$ has  $t$ inner automorphisms extending $id_N$, with
 $$\frac{q^l-1}{q-1}\leq t\leq \frac{q^n-1}{q-1}.$$

%
%

%
%

\section{Nonassociative  cyclic algebras} \label{sec:nonasscyclicfinite}

In this section unless specifically noted otherwise, let
$$f(t)=t^m-a\in K[t;\sigma], \quad  a\in K\setminus F$$
be irreducible (which is always the case if $a$ belongs to no proper subfield of $K/F$), $\sigma$ have order $n=m$ and let
 $$A=(K/F,\sigma,a)=K[t;\sigma]/K[t;\sigma](t^m-a).$$

Then $A$ is an example of a Sandler semifield \cite{San62}. $A$ is also called a \emph{nonassociative cyclic (division) algebra
 of degree $m$}, because  its construction (hence its multiplication) is similar to the one of an associative cyclic algebra which is defined by
 $ K[t;\sigma]/ K[t;\sigma](t^m-a)$ but choosing $a\in F^\times$. For $m=n=2$, $A$ is also called a \emph{nonassociative
 quaternion algebra} and was first described by Dickson \cite{D06}.
We know that $\mathbb{N}_l(A)=\mathbb{N}_m(A)=\mathbb{N}_r(A)=K\cdot 1$.
Moreover,
$$(K/F,\sigma,a)\cong (K/F,\sigma,b)$$
 if and only if
$$\sigma^i(a) = kb \text{ for some } 0\leq i \leq m-1 \text{ and some }k \in F^{\times}$$
 \cite[Corollary 34]{BP}.

By Theorem \ref{cor:innerII}, $A$ has exactly $s=(q^m-1)/(q-1)$ inner
automorphisms, all of them extending $id_K$. These are given by the
 $F$-automorphisms $H_{id,l}$   for all $l\in K$ such that $N_{K/F}(l)=1$.
  The subgroup they generate is cyclic and isomorphic to ${\rm ker}(N_{K/F})$.

\subsection{}

\begin{theorem} (\cite[Theorem 22]{BP}) \label{thm:aut1II}
 Suppose  $m$ divides $ (q-1)$ and
let $\omega$ denote a non-trivial $m$th root of unity in $F$.
\\ (i)  $\langle H_{id,\omega}\rangle$
is a cyclic subgroup of ${\rm Aut}_F(A)$ of order at most $m$.
 If $\omega$ is a primitive $m$th root of unity, then $\langle H_{id,\omega}\rangle$
has order $m$.
\\ (ii)  Suppose $N_{K/F}(l)=\omega$ is a primitive  $m$th root of unity and
$\sigma(a)= \omega a$. Then the subgroup generated by
$H=H_{\sigma,l}$ has order $m^2$.
\\ (iii) For each $m$th root of unity $\omega\in F$, $l\in K$ with $N_{K/F}(l)=\omega$ and
a $j\in \{1,\dots, m-1\}$ such that $\sigma^j(a)=\omega a$, there
 is an automorphism $H_{\sigma^{j},l}$ extending $\sigma^{j}$.
\end{theorem}

\begin{proposition} (\cite[Theorem 21]{BP})\label{cor:aut_nonass_cyclicII}
A Galois automorphism $\sigma^j\not=id$ can be extended to an automorphism $H\in {\rm Aut}_F(A)$
if and only if there is some $l \in K$ such that
 $$\sigma^j(a)= N_{K/F}(l) a.$$
In that case, $H=H_{\sigma^j,l}$  and if $m$ is prime then  $N_{K/F}(l)=\omega$ is a primitive
$m$th root of unity
 and   there exist  $s =(q^m-1)/(q-1)$ such extensions.
\end{proposition}

\begin{theorem}  \cite[Theorem 24]{BP} \label{thm:automorphism_of_Sf_field_caseIII}
Let $K/F$ have prime degree $m$. Suppose that $F$ contains a primitive $m$th root of unity, where $m$ is coprime to the characteristic of $F$
and so $K = F(d)$ where $d$ is a root of an irreducible polynomial $t^m - c \in F[t]$. Then $H$ is an
automorphism of $A$ extending $\sigma^j\not=id$ if and only if
$H=H_{\sigma^j , k}$ for some $k \in K^{\times}$,
 where
$N_{K/F}(k)$
 is a primitive $m$th root of unity and
$a=cd^l$ for some $c\in F^\times$ and some power $d^{l}$.
\end{theorem}

For more general polynomials than $f$ this yields:

\begin{corollary}
Suppose that $F$ contains a primitive $m$th root of unity, where $m$ is coprime to the characteristic of $F$
and so $K = F(d)$ where $d$ is a root of some $t^m - c \in F[t]$. Let
$$g(t) = t^m -\sum_{i=0}^{m-1} a_i t^i \in K[t;\sigma]$$
and $a_{0}\in K\setminus F$, such that $a_0\not=c d^i$ for any $0 \leq i \leq m-1$, $c\in F^\times$.
  Then every $F$-automorphism of $S_g$ leaves $K$ fixed, is inner and
  ${\rm Aut}_F(S_g)\subset {\rm ker}(N_{K/F})$
   is a subgroup, thus cyclic with at most $s=(q^m-1)/(q-1)$ elements.
 \end{corollary}
  This follows from \cite[Corollary 25]{BP}.

\begin{corollary}
Suppose that $F$ does not contain an $m$th root of unity (i.e., $m $ and $(q-1)$ are coprime). Let
$$g(t) = t^m -\sum_{i=0}^{m-1} a_i t^i \in K[t;\sigma]$$ and $a_{0}\in K\setminus F$.
 Then every $F$-automorphism of $S_g$ leaves $K$ fixed, is inner and
 ${\rm Aut}_F(S_g)$ is isomorphic to a subgroup of $ {\rm ker}(N_{K/F})$, thus cyclic with at most $s=(q^m-1)/(q-1)$ elements.
   In particular, if ${\rm ker}(N_{K/F})$ has prime order, then either
${\rm Aut}_F(S_g)$ is trivial or ${\rm Aut}_F(S_g)\cong {\rm ker}(N_{K/F})$.
 \end{corollary}

We can also rephrase our results as follows:

\begin{proposition}\label{prop:t^m-a_automorphism_field finite}
Let $\alpha$ be a primitive element of $K$, i.e. $K^{\times} =
\langle\alpha\rangle$.
\\ (i) $\langle G_{\alpha}\rangle$ is a cyclic subgroup of ${\rm Aut}_F(A)$ of order $s =(q^m-1)/(q-1)$,
containing inner automorphisms.
\\  (ii) Suppose one of the following holds:
\\ (a)
 $m $ and $(q-1)$ are coprime.
\\ (b)  $m$ is prime and $F$  a field where $m$
is coprime to the characteristic of $F$, containing a primitive $m$th
root of unity. Let $K = F(d)$ be a cyclic field extension of $F$ of
degree $m$. Let  $a \in K \setminus F$ and
 $a\not=\lambda d^i$ for every $i \in\{ 0, \ldots, m-1\}$, $\lambda\in F^\times$.
\\ Then
${\rm Aut}_F(A)=\langle G_{\alpha}\rangle.$
\end{proposition}

\begin{proof}
If $K^{\times} = \langle\alpha\rangle$ then $F^{\times} = \langle\alpha^s\rangle$ for $s =(q^m-1)/(q-1)$.
 In particular $\alpha^s \in F^{\times}$ but $\alpha^j \notin F$ for all smaller $j$.
 The result now follows from   \cite[Theorem 21(iii)]{BP}
\end{proof}

For more general choices of twisted polynomials this means:

\begin{theorem}
With the assumptions of Proposition \ref{prop:t^m-a_automorphism_field finite} (ii) on
$K/F$ and $a$, for each irreducible
 $$g(t)=t^m-\sum_{i=0}^{m-1}a_it^i \text{ with }a_0=a\in K\setminus F,$$
${\rm Aut}_F(S_g)$ is a subgroup of ${\rm ker}(N_{K/F})$ and
therefore cyclic of order at most $s =(q^m-1)/(q-1)$.
\end{theorem}

This is a consequence of Theorem \ref{Aut(S_f) subgroup corollary}.

\subsection{The automorphism groups of some nonassociative cyclic algebras}

 In this subsection, we assume that  $F$ is a field where $m$ is coprime to the
characteristic of $F$, and that $F$ contains a primitive $m$th root of unity
$\omega$, so that $K = F(d)$. Let $s=(q^m-1)/(q-1)$.

\begin{lemma} \label{lem:n divides s}
Suppose $m \vert (q-1)$ then:
\\ (i)  $m \vert s$.
\\ (ii)  If $m$ is odd then $m^2 \nmid (ls)$ for all $l \in \{ 1, \ldots, m-1 \}$.
\\ (iii) If $(q-1)/m$ is even then $m^2 \nmid (ls)$ for all $l \in \{ 1, \ldots, m-1 \}$.
\end{lemma}

\begin{proof}
(i) We prove first that
\begin{equation} \label{eqn:Lemma:n divides s I}
(q-1) \vert \Big( \big( \sum_{i=0}^{m-1} q^i \big) - m \Big)
\end{equation}
for all $m \geq 2$ by induction:

Clearly \eqref{eqn:Lemma:n divides s I} holds for $m = 2$. Suppose \eqref{eqn:Lemma:n divides s I} holds for some $m \geq 2$, then
\begin{equation} \label{eqn:Lemma:n divides s II}
\big( \sum_{i=0}^{m} q^i \big) - (m+1) =  \big( \sum_{i=0}^{m-1} q^i \big) - m + q^m - 1 =  \big( \sum_{i=0}^{m-1} q^i \big) - m +  \big( \sum_{i=0}^{m-1} q^i \big) (q-1).
\end{equation}
Now, $(q-1) \vert \Big( \big( \sum_{i=0}^{m-1} q^i \big) - m \Big)$ and so \eqref{eqn:Lemma:n divides s I} holds by induction. In particular
$$m \vert \Big( \big( \sum_{i=0}^{m-1} q^i \big) - m \Big),$$
therefore $m$ divides
$\big( \sum_{i=0}^{m-1} q^i \big) - m + m = s$
as required.
\\ (ii) and (iii):
Write $q = 1 + rm$ for some $r \in \mathbb{N}$, then
\begin{align*}
q^j &= (1+rm)^j = \sum_{i=0}^j \binom{j}{i} (rm)^i \equiv \sum_{i=0}^1 \binom{j}{i} (rm)^i \ \mathrm{ mod} \ (m^2) \\ & \equiv (1 + jrm) \ \mathrm{ mod} \ (m^2)
\end{align*}
for all $j \geq 1$. Therefore
\begin{align*}
ls &= l \sum_{j=0}^{m-1} q^j \equiv l \big( 1 + \sum_{j=1}^{m-1}(1+jrm) \big) \ \mathrm{ mod} \ (m^2) \\ & \equiv \Big( lm + lrm \frac{(m-1)m}{2} \Big) \ \mathrm{ mod} \ (m^2),
\end{align*}
for all $l \in \{ 1, \ldots, m-1 \}$.
If $m$ is odd or $r=(q-1)/m$ is even then
$$\frac{lr(m-1)}{2} \in \mathbb{Z}$$
which means
$$ls \equiv lm \ \mathrm{ mod} \ (m^2) \not\equiv 0 \ \mathrm{ mod} \ (m^2),$$
that is, $m^2 \nmid (ls)$ for all $l \in \{ 1, \ldots, m-1 \}$.
\end{proof}

 Recall that the semidirect product
$$\mathbb{Z} / m \mathbb{Z} \rtimes_l
\mathbb{Z} / n \mathbb{Z} = \langle x,y \,|\,x^m = 1, y^n= 1, yxy^{-1}=x^l\rangle$$
of $\mathbb{Z} / m \mathbb{Z}$ and $\mathbb{Z} / n \mathbb{Z} $
corresponds to the choice of an integer $l$ with $l^n
\equiv 1 \text{ mod } m$.
 Let $A = (K/F,\sigma,a)$ where $a = \lambda d^i$ for some
$i \in \{ 1 , \ldots, m-1 \}$, $\lambda \in F^{\times}$.

\begin{theorem} \label{thm:Automorphisms of nonassociative cyclic algebras over finite fields}
Suppose $m$ is odd or $r=(q-1)/m$ is even. Then $\mathrm{Aut}_F(A)$ is a group of order
$ms$ and contains a subgroup isomorphic to the semidirect product
\begin{equation} \label{eqn:Automorphisms of nonassociative cyclic algebras over finite fields semidirect not nec prime}
\mathbb{Z} / \Big( \frac{s}{m} \Big) \mathbb{Z} \rtimes_{q} \mathbb{Z} / (m \mu) \mathbb{Z},
\end{equation}
where $\mu = m/\mathrm{gcd}(i,m)$.
Moreover, if $i$ and $m$ are coprime, then
\begin{equation} \label{eqn:Automorphisms of nonassociative cyclic algebras over finite fields semidirect not nec prime 2}
\mathrm{Aut}_F(A) \cong \mathbb{Z} / \Big( \frac{s}{m} \Big) \mathbb{Z} \rtimes_{q} \mathbb{Z} / (m^2) \mathbb{Z}.
\end{equation}
\end{theorem}

\begin{proof}
Let $\tau: K \rightarrow K, \ k \mapsto k^q$, then
$$\tau^j(a) = \omega^{ij}a,$$
for all $j \in \{ 0, \ldots, m-1 \}$ where $\omega \in F^{\times}$ is a primitive
$m^{\text{th}}$ root of unity by \cite[Lemma 23]{BP}.
As $\tau$ generates $\mathrm{Gal}(K/F)$, the automorphisms of $A$ are precisely the maps $H_{\tau^j,k}$,
where $j \in \{ 0, \ldots, m-1 \}$ and $k \in K^{\times}$ are such that $\tau^j(a) = N_{K/F}(k)a$ by
Proposition \ref{cor:aut_nonass_cyclicII}. Moreover there are exactly $s$ elements $k \in K^{\times}$ with
$N_{K/F}(k) = \omega^{ij}$ by Proposition \ref{cor:aut_nonass_cyclicII}, and each of these elements corresponds to a unique automorphism of $A$. Therefore $\mathrm{Aut}_F(A)$ is a group of order $ms$.

Choose $k \in K^{\times}$ such that $N_{K/F}(k) = \omega^i$ so that $H_{\tau,k} \in \mathrm{Aut}_F(S_f)$.
 As $\tau$ has order $m$, $H_{\tau,k} \circ \ldots \circ H_{\tau,k}$ ($m$-times) becomes
 $H_{id,b}$ where $b = \omega^i = N_{K/F}(k)$. Notice $\omega^i$ is a primitive $\mu^{\text{th}}$ root
 of unity where $\mu = m/\mathrm{gcd}(i,m)$, then $H_{id,b}$ has order $\mu$ and so the subgroup of
 $\mathrm{Aut}_F(S_f)$ generated by $H_{\tau,k}$ has order $m \mu$.

$\langle G_{\alpha}\rangle$ is a cyclic subgroup of $\mathrm{Aut}_F(S_f)$ of order $s$ by  Proposition \ref{prop:t^m-a_automorphism_field finite}
where $\alpha$ is a primitive element of $K$. Furthermore, $m \vert s$ by Lemma \ref{lem:n divides s} and so
$\langle G_{\alpha^m}\rangle$ is a cyclic subgroup of $\mathrm{Aut}_F(A)$ of order $s/m$. We will prove $\mathrm{Aut}_F(A)$
contains the semidirect product
\begin{equation} \label{eqn:form of <H_sigma,k> automorphisms of nonassociative cyclic over finite fields 2}
\langle G_{\alpha^m}\rangle \rtimes_q \langle H_{\tau,k}\rangle:
\end{equation}
The inverse of $H_{\tau,k}$ in $\mathrm{Aut}_F(A)$ is $H_{\tau^{-1},\tau^{-1}(k^{-1})}$ and a tedious calculation shows that
$$H_{\tau,k} \circ G_{\alpha^m} \circ H_{\tau,k}^{-1} = G_{\alpha^{mq}} = (G_{\alpha^m})^q.$$

Notice $q^m = qs -s + 1$, i.e. $q^m \equiv 1 \text{ mod } s$, and so $q^{m \mu} \equiv 1 \text{ mod } s$. Then
$m \vert s$ by Lemma \ref{lem:n divides s}, hence $q^{m \mu} \equiv 1 \text{ mod } (s/m)$. In order to prove
\eqref{eqn:form of <H_sigma,k> automorphisms of nonassociative cyclic over finite fields 2}, we are left to show that
$\langle H_{\tau,k}\rangle \cap \langle G_{\alpha^m}\rangle = \{ \mathrm{id} \}$.

Suppose for contradiction $\langle H_{\tau,k}\rangle \cap \langle G_{\alpha^m}\rangle \neq \{ \mathrm{id} \}$, then
$H_{id,\omega^l} \in \langle G_{\alpha^m}\rangle$ for some $l \in \{ 1, \ldots, m-1 \}$. Therefore $\langle G_{\alpha^m}\rangle$
contains a subgroup of order $m/\mathrm{gcd}(l,m)$ generated by $H_{id, \omega^l}$ and so
$(m/\mathrm{gcd}(l,m)) \vert (s / m)$. This means $m^2 \vert (s \mathrm{gcd}(l,m))$, a contradiction by Lemma
\ref{lem:n divides s}.

Therefore $\mathrm{Aut}_F(A)$ contains the subgroup
$$\langle G_{\alpha^m}\rangle \rtimes_q \langle H_{\tau,k}\rangle \cong \mathbb{Z} / \Big( \frac{s}{m} \Big) \mathbb{Z} \rtimes_{q} \mathbb{Z} /
(m \mu) \mathbb{Z}.$$
If $\mathrm{gcd}(i,m) = 1$ this subgroup has order $ms$ and since $|\mathrm{Aut}_F(A)| = ms$,
this is all of $\mathrm{Aut}_F(A)$.
\end{proof}

\begin{theorem} \label{thm:semidirect}
Suppose  $m$ is prime
and $m \vert (q-1)$.
\begin{itemize}
\item[(i)] If $m = 2$ then $\mathrm{Aut}_F(A)$ is the dicyclic group $\mathrm{Dic}_r$ of order $4r = 2q + 2$.
\item[(ii)] If $m > 2$ then
\begin{equation} \label{eqn:Automorphisms of nonassociative cyclic algebras over finite fields semidirect}
\mathrm{Aut}_F(A) \cong \mathbb{Z} / \Big( \frac{s}{m} \Big) \mathbb{Z} \rtimes_{q} \mathbb{Z} / (m^2) \mathbb{Z}.
\end{equation}
\end{itemize}
\end{theorem}

\begin{proof}
(i) We already know that ${\rm Aut}_F(A)$ has order $ 2(q+1)$. Let $\alpha \in K$ be a primitive element.
Then $\langle G_{\alpha}\rangle$ is a subgroup of ${\rm Aut}_F(A)$ of order $s$ by Proposition
\ref{prop:t^m-a_automorphism_field finite}.
 Furthermore, since $\sigma(a) = -a$, there  are precisely $s = q+1$ automorphisms $H_{\sigma,k}$ where
 $k \in K$ is such that
 $N_{K/F}(k) = -1$.
Pick any such $k \in K$. Then an easy calculation shows that ${\rm Aut}_F(A) \cong Dic_r$, i.e. that
 $${\rm Aut}_F(A) = \langle H_{\sigma,k}, G_{\alpha} \ \vert \ G_{\alpha}^{2r} = 1, \ H_{\sigma,k}^2 = G_{\alpha}^r,
 \ H_{\sigma,k}^{-1} G_{\alpha} H_{\sigma,k} = G_{\alpha}^{-1} \rangle.$$
 (ii) follows immediately from Theorem \ref{thm:Automorphisms of nonassociative cyclic algebras over finite fields}.
\end{proof}

Recall that the smallest dicyclic group ${\rm Dic}_2$ of order $4r=8$ (this only occurs if $q=3$)
is isomorphic to the quaternion group.
More generally, when $r$ is a power of 2, the dicyclic group ${\rm Dic}_r$ of order
$4r=2q+2$ is isomorphic to the generalized quaternion group.

 Note that if $m=2$ and $4$ divides $s=q+1$ then ${\rm Aut}_F(A)$ is not a semidirect product,
 since in this case  $\langle H_{\sigma,k}\rangle\cap \langle G_{\alpha^2}\rangle\not=\{id\}$.

\begin{remark}
Let $F$ have  characteristic not 2 and $K= F(d)$ be a quadratic field extension of $F$, then $A=(K/F, \sigma, a)$ is
a \emph{nonassociative quaternion algebra}. Nonassociative quaternion algebras are up to isomorphism the only
  proper semifields of order $q^4$ with center $F1$ and nucleus containing $K \cdot 1$
\cite[Theorem 1]{W}. Although being not associative, they are closely related to associatve quaternion algebras, as their multiplication can be seen as a
 canonical generalization of the classical Cayley-Dickson doubling process used to construct quaternion algebras
 out of a separable quadratic field extension \cite{AP}. If $a\not=\lambda d$ for any  $\lambda\in F^\times$, then
 ${\rm Aut}_F(A)\cong \mathbb{Z}/(q+1) \mathbb{Z}$
  and all automorphisms are inner (as it is the case for classical quaternion algebras).
 If $a=\lambda d$ for some  $\lambda\in F^\times$, then ${\rm Aut}_F(A)$ is
the dicyclic group of order $2q+2$ (Theorem \ref{thm:semidirect}).
\end{remark}

%
%

\section{The automorphisms of Hughes-Kleinfeld and Knuth semifields} \label{sec:knuth}

Let $K/F$ be a  Galois field extension of degree $n$. Choose $\eta,
\mu\in K$ and a nontrivial automorphism $\sigma\in {\rm Aut}_F(K)$.
For  $x,y,u,v \in K$ the following four multiplications make the
$F$-vector space $K \oplus K$ into an algebra over $F$:
\begin{align*}
Kn_1: (x,y) \circ (u,v) &= (xu + \eta \sigma(v)\sigma^{-2}(y) , vx +y
\sigma(u) + \mu \sigma(v)\sigma^{-1}(y) ),\\
 Kn_2: (x,y) \circ(u,v) &= (xu + \eta \sigma^{-1}(v) \sigma^{-2}(y) , vx + y \sigma(u)+ \mu v\sigma^{-1}(y)),\\
 Kn_3: (x,y) \circ (u,v) &= (xu + \eta\sigma^{-1}(v)y , vx + y \sigma(u) + \mu vy),\\
 HK: (x,y) \circ (u,v) &= ( x u + \eta y \sigma(v) ,  x v + y \sigma(u) + \mu y \sigma(v)).
\end{align*}
The unital algebras  given by each of the above multiplications are
denoted $Kn_1(K, \sigma, \eta, \mu)$, $Kn_2(K, \sigma, \eta, \mu)$,
$Kn_3(K, \sigma, \eta, \mu)$ and $HK(K, \sigma, \eta, \mu)$,
respectively.  The first three algebras were defined by Knuth
 and the last one  by Hughes and Kleinfeld \cite{HK}, \cite{Kn65}. If $\sigma^2 = id$ and $\mu = 0$, they are the same algebra with multiplication
$(x,y) \circ (u,v) = (xu + \eta y \sigma(v), xv + y \sigma(u)).$
Each of the  algebras is a division algebra if and only if
$$f(t)= t^2-\mu t-\eta\in K[t;\sigma]$$ is irreducible \cite{HK}, \cite{Kn65}.  For
$F=\mathbb{F}_q$, $K=\mathbb{F}_{q^n}$ and irreducible $f(t)$ (i.e.
$\eta\not=0$), we thus obtain  semifields. Identifying $(u,v)$ with $u+tv$, we see that the Hugh-Kleinfeld algebra
$$HK(K, \tau,\eta,\mu)=S_f  \text{ with } f(t)= t^2-\mu t-\eta\in K[t;\tau]$$
is a Petit algebra and that
$$Kn_2(K, \sigma, \eta,\mu)=_fS  \text{ with } f(t)= t^2-\mu t-\eta\in K[t;\sigma],$$
hence is the opposite algebra of a suitable Petit algebra. Thus $Kn_2(K, \sigma, \eta,\mu)=S_g$ for $g(t)= t^2-\mu' t-\eta' \in
K[t;\sigma^{-1}]$ for some suitable $\mu' ,\eta'\in K$ by \cite[Corollary 4]{LS}. Thus
 the automorphisms for any $Kn_2(K, \sigma, \eta,\mu) $
will be the same as for a Petit algebra
 $ HK(K, \sigma^{-1}, \eta',\mu')$.

 Suppose that either $\sigma^2 \neq id$ or that $\mu \neq 0$.
Then the following is well-known (cf. \cite{HK}, \cite{Kn65}):
\begin{itemize}
\item $K \cdot 1$ is not contained in the left, right or middle nucleus of $Kn_1(L, \sigma, \eta, \mu)$.
\item $\mathbb{N}_m(A)=\mathbb{N}_r(A)=K \cdot 1$ and $\mathbb{N}_l(A)\cong\mathbb{F}_{q^2}$  for $A=Kn_2(K, \sigma,\eta, \mu)$
\item $\mathbb{N}_l(A)=\mathbb{N}_r(A)=K \cdot 1$ for  $A=Kn_3(K, \sigma, \eta, \mu)$ but $K \cdot 1$ is not contained in the middle nucleus.
\end{itemize}

Hence $Kn_1(K, \sigma, \eta, \mu)$, $Kn_2(K, \sigma, \eta, \mu)$,
$Kn_3(K, \sigma, \eta, \mu)$ and $HK(K, \sigma, \eta, \mu)$ are mutually
non-isomorphic algebras.

We now describe all automorphisms for the algebras  $HK(K,
\sigma, \eta, \mu)$ (hence also for $Kn_2(K, \sigma, \eta, \mu)$) and $Kn_3(K,
\sigma, \eta, \mu)$. We also exhibit some automorphisms for the
algebra  $Kn_1(K, \sigma, \eta, \mu)$. This complements and improves the results in \cite{W09}.

\begin{theorem}
(i) \label{hkauto}
 All automorphisms of the Petit algebra $A=HK(K, \sigma, \eta, \mu)$ are of the form
$$H_{\tau,k}(x_0+x_1t)= \tau(x_0)+ k\tau(x_1)t$$
where $\tau \in {\rm Aut}_F(K)$ and $k \in K^\times$ such that
$ \eta k \sigma(k) = \tau(\eta) \text{ and } \mu \sigma(k) = \tau(\mu).$
\\ (ii) \label{kn3auto}
All automorphisms of $Kn_3(K, \sigma, \eta, \mu)$ are of the form
$$H_{\tau,k}(x_0+x_1t)= \tau(x_0)+ k\tau(x_1),$$
where $\tau \in {\rm Aut}_F(K)$ and $k \in K^{\times}$  such that
$ \eta \sigma^{-1}(k) \sigma^{-2}(k) = \tau(\eta)$ and  $\mu \sigma^{-1}(k) = \tau(\mu)$.
\\ In both (i) and (ii),  $N_{K/F}( k) =\pm 1 $ and if $\mu\not=0$, even  $N_{K/F}(k) = 1 $.
\end{theorem}

\begin{proof}
(i) This follows from the results mentioned in Subsection
\ref{sec:semifields}, i.e. Theorem \cite[Theorem 4]{BP}. Furthermore, $\eta k
\sigma(k) = \tau(\eta) $ implies $N_{K/F}(\eta k^2) =N_{K/F}( \eta)
$, i.e $N_{K/F}( k^2) =N_{K/F}( k)^2=1 $ since $\eta\not=0$, thus
$N_{K/F}( k) =\pm 1 $. If $\eta\in F^\times$ then $\eta k \sigma(k) =
\tau(\eta) $ yields $\eta k \sigma(k) = \eta $, hence $ k \sigma(k) =
1 $. The equation  $ \mu \sigma(k) = \tau(\mu)$ implies $N_{K/F}( \mu
k) = N_{K/F}(\mu)$, i.e $N_{K/F}( k) = 1 $ for $\mu\not=0$.
\\ (ii) Since any automorphism  preserves the left nucleus $K \cdot 1$, it follows that $H \vert_K = \tau$ for some
 $\tau \in \text{Aut}_F(K)$.
Although here we are not dealing with a Petit algebra, (ii) is now proved  analogous to (i) with the same arguments as
used in the proof of  \cite[Theorem 4]{BP}, since
comparing coefficients also yields $H(t) = kt$ for some $k \in K^{\times}$.
\end{proof}

\begin{corollary}
Let $\mu\in F^\times$ and $A$ be either $HK(K, \sigma, \eta,\mu)$,  $ Kn_2(K, \sigma, \eta, \mu)$
 or $Kn_3(K, \sigma,\eta,\mu)$.
\\ (i) If $\eta\in K\setminus F$ then
${\rm Aut}_F(A)=\{id\}.$
\\ (ii)  If $\eta\in F$ and $f(t)= t^2-\mu t-\eta$ is not right invariant then
$${\rm Aut}_F(A)\cong  \mathbb{Z}/n\mathbb{Z}.$$
\end{corollary}

\begin{proof}
 Theorem \ref{thm:automorphism_of_Sf_field_caseIV} for $S_f$ implies
 the statement for the first two types.
The argument for the third type is analogous: if $\mu\in F^\times$ then $k=1$, thus
$\eta  = \tau(\eta) $ forces $\tau=id$ or $\eta\in F^\times$. If $\eta\in K\setminus F$ thus $\tau=id$
and ${\rm Aut}_F(A)=\{id\}.$
 If  $\eta\in F$ and $f(t)$ is not right invariant then
${\rm Aut}_F(A)\cong {\rm Gal}(K/F) \cong \mathbb{Z}/n\mathbb{Z}.$
\end{proof}

\begin{proposition} \label{prop:last}
Let  $A$ be one of the algebras $HK(K, \sigma, \eta, \mu), Kn_2(K,
\sigma, \eta, \mu)$ or $Kn_3(K, \sigma, \eta, \mu)$ where $\mu \neq
0$. Then
\[{\rm Aut}_F(A) \cong \{\tau \in {\rm Gal}(K/F) \mid \tau \Bigg( \frac{\mu \sigma(\mu)}{\sigma(\eta)}\Bigg) =
\frac{\mu \sigma(\mu)}{\sigma(\eta)}\} \text{ via } H_{\tau,k} \mapsto
\tau.\]

\begin{proof}
Suppose for instance $A = HK(K, \sigma, \eta, \mu)$. Take the
automorphism $H_{\tau,k}$. By Proposition \ref{hkauto},
$\mu \sigma(k) = \tau(\mu)$ and $\eta b \sigma(k) = \tau(\eta)$.
(Note that since $\mu \neq 0$, the element $k \in K$ is determined
completely by the action of $\tau$ on $\mu$.) Substituting in $ k =
\sigma^{-1}(\tau(\mu)) \sigma^{-1}(\mu)^{-1}$ and rearranging gives
$\sigma(\eta) \tau(\mu)\sigma(\tau(\mu)) = \sigma(\tau(\eta))\mu \sigma(\mu).$
This implies
\[\tau \Bigg(\frac{\mu \sigma(\mu)}{\sigma(\eta)}\Bigg) = \frac{\mu \sigma(\mu)}{\sigma(\eta)}.\]
\end{proof}
\end{proposition}

For $Kn_1(K, \sigma, \eta, \mu)$, $K \cdot 1$ is not contained in any of the
nuclei. However, if we assume that an automorphism of $Kn_1(K, \sigma, \eta, \mu)$
restricts to an automorphism of $K$, then it must be of a similar
form to the above automorphisms:

\begin{proposition}
Suppose $H$ is an automorphism of $A= Kn_1(L, \sigma, \eta, \mu)$
which restricts to an automorphism $\tau \in {\rm Aut}_F(K)$. Then
$$H(x_0+x_1t)=\tau(x_0)+k\tau(x_1)t$$
for some $k\in K^\times$, such that $\eta \sigma^{-1}(k)
\sigma^{-2}(k) = \tau(\eta)$ and $\mu \sigma(b) \sigma^{-1}(k) =
\tau(\mu) k$. In particular, $N_{K/F}(k) =\pm 1 $ and if
$\mu\not=0$, $N_{K/F}(k) = 1 $. If $\eta\in F^\times$ then $
\sigma^{-1}(k) \sigma^{-2}(k) = 1$.
\end{proposition}

The proof  is  similar to that of Proposition \ref{hkauto}.

%
%

\section{Isomorphisms between semifields } \label{sec:isofinite}

\subsection{}

If $K$ and $L$ are finite fields and $$S_f=K[t;\sigma]/K[t;\sigma]f(t)\cong
L[t;\sigma']/L[t;\sigma']g(t)=S_g$$
two isomorphic  Jha-Johnson semifields with
$f\in K[t; \sigma]$ and $g\in L[t;\sigma']$ both monic, irreducible and not right invariant, then
$$K\cong L,  \quad {\rm deg}(f)={\rm deg}(g) \text{  and  }{\rm Fix}(\sigma)\cong {\rm Fix}(\sigma'),$$
since isomorphic algebras have the same dimensions, and isomorphic nuclei and center.

Moreover, if $G$
is an automorphism of $R =K[t;\sigma]$, $f(t) \in R$ is irreducible and $g(t) = G(f(t))$, then $G$ induces
an isomorphism $S_f \cong S_{g}$ \cite[Theorem 7]{LS}.  In the following we focus on the situation that
 $F = \mathbb{F}_q$,  $K = \mathbb{F}_{q^n}$, $\text{Gal}(K/F) =
\langle\sigma\rangle$, and use
 $$f(t) = t^m - \sum_{i=0}^{m-1} a_i t^i ,\quad g(t) =
t^m - \sum_{i=0}^{m-1} b_i t^i \in K[t;\sigma].$$

\cite[Theorems  28 and 29]{BP} yield in this setting a generalization of  \cite[Theorem 4.2 and 5.4]{W00}
which proved this statement only for $ m=2,3$.:

\begin{theorem}  \label{general_isomorphism_theorem}
(i) Suppose  $n \geq m-1$. Then $S_f \cong S_g$ if and only if there exists
$\tau\in {\rm Gal}(K/F)$ and $k \in K^{\times}$ such that
\begin{equation} \label{equ:necII}
\tau(a_i) = \Big( \prod_{l=i}^{m-1} \sigma^l(k) \Big) b_i
\end{equation}
 for all $i \in\{ 0, \ldots, m-1\}$. Every such $\tau$ and $k$ yield a
unique isomorphism $G_{\tau,k}: S_f \rightarrow S_g$,
 $$G_{\tau,k}( \sum_{i=0}^{m-1}x_i t^i )= \tau(x_0) +  \sum_{i=1}^{m-1}
\tau(x_i) \prod_{l=0}^{i-1}\sigma^l(k) t^i.$$
(ii)
Suppose  $n<m-1$. If there exists $\tau\in {\rm Gal}(K/F)$ and $k \in
K^{\times}$ such that Equation (\ref{equ:necII}) holds for all $i \in\{ 0, \ldots, m-1\}$ then
$S_f \cong S_g$ with an isomorphism  $G_{\tau,k} : S_f\rightarrow S_g$ as in (i).
\end{theorem}

As a direct consequence of Theorem
\ref{general_isomorphism_theorem} we obtain:

\begin{corollary} \label{thm:condition_f_g}
 Let  $n \geq m-1$.
\\ (i) If $S_f \cong S_g$ then $a_i=0$ if and only if $b_i=0$, for all $i \in\{ 0, \ldots, m-1\}$.
\\ (ii) If there exists an  $i \in\{ 0, \ldots, m-1\}$ such that $a_i=0$ but $b_i\not=0$ or vice versa, then
$S_f \not\cong S_g$.
\end{corollary}

\cite[Corollaries 33, 34]{BP} yield for instance:

\begin{corollary}
Suppose  $n \geq m-1$ and that one of the following holds:
\\ (i) There exists  $i \in\{ 0, \ldots, m-1\}$ such that $b_i\not=0$ and
$$N_{K/F}(a_i b_i^{-1}) \not\in  F^{\times (m-i)};$$
 (ii) $N_{K/F}(a_0) \neq N_{K/F}(b_0)$ in $F^{\times} /F^{\times m}$;
\\ (iii) $b_{m-1}\not=0$ and $N_{K/F}(a_{m-1}b_{m-1}^{-1}) \not\in  F^\times$;
\\ (iv) $m=n$, $a_0\in F^\times$ and $b_0\in K\setminus F$.
\\ Then $S_f \ncong S_g$.
\end{corollary}

\begin{corollary} \label{Isomorphism cyclic extension}
Let $n =m$, $f(t) = t^m - a,$ $g(t) = t^m - b \in K[t;\sigma]$ where $a, b\in K \setminus F$.
\\ (i) $S_f \cong S_g$ if and only if
there exists $\tau\in {\rm Gal}(K/F)$ and $k \in K^{\times}$ such that
$$\tau(a) = N_{K/F}(k) b.$$
 (ii) If $a \not=\alpha b$ for all $\alpha \in F^{\times}$ or if
 $N_{K/F}(a/b)\not\in F^{\times m}$ then $S_f \not\cong S_g$.
\end{corollary}

\subsection{The isomorphism classes of nonassociative cyclic algebras of prime degree} \label{subsec:isoclassescyclic}

As an example, we count how many nonisomorphic semifields $(K/F, \sigma, a)$ there are for a given field extension $K/F$.

\begin{example}
Let $F= \mathbb{F}_2$ and let $K = \mathbb{F}_4$, then we can write
$K = \{0,1,x, 1+x \}$
where $x^2+x+1 = 0$. Thus for  $(K/F,\sigma,a)$ we can either choose
$a = x$ or $a = 1+x$. Both choices will give a division algebra. We
also know that $(K/F,\sigma,a) \cong (K/F,\sigma,b)$ if and only if
$\sigma(a) = N_{K/F}(l)b$  or $a = N_{K/F}(l)b$. $N_{K/F}: L^{\times} \rightarrow
F^{\times}$ is surjective, so $N_{K/F}(l) = 1$ for all $l \in
K^\times$. The statement then reduces to $(K/F,\sigma,a) \cong
(K/F,\sigma,b)$ if and only if $\sigma(a) = b$  or $a = b$. Now
\[\sigma(x) = x^2 = 1+x.\]
Therefore $(K/F,\sigma,x) \cong (K/F,\sigma,1+x)$, so there is only
one nonassociative (quaternion) algebra up to isomorphism which can be constructed
using $K/F$. Its automorphism group consists of inner automorphisms and is
isomorphic to $\langle G_x \rangle\cong \mathbb{Z}/3\mathbb{Z}$.

\end{example}

More generally we obtain:

\begin{theorem} \label{numb}
 (i)  If $m$  does not divide $q-1$
 then there are exactly
\[\frac{q^m-q}{m(q-1)}\]
non-isomorphic  semifields $(K/F, \sigma, a)$ of  degree $m$.
\\ (ii)
 If $m$ divides $q-1$ and is prime then there are exactly
\[m-1 + \frac{q^m-q - (q-1)(m-1)}{m(q-1)}\]
non-isomorphic  semifields $(K/F, \sigma, a)$ of  degree $m$.

\end{theorem}

\begin{proof}
 Define an equivalence relation on the set $K
\setminus F$ by
\[a \sim b \text{ if and only if } (K/F, \sigma, a) \cong (K/F, \sigma, b).\]

For each $a \in K\setminus F$ we have
\[(K/F, \sigma, a) \cong (K/F, \sigma, \sigma^i(a))\]
for $0 \leq i \leq m-1$ and
\[(K/F, \sigma, a) \cong (K/F, \sigma, ka)\]
for $k \in F^{\times}$. If the elements $k\sigma^i(a)$ for $0 \leq i
\leq m-1$ and $k \in F^{\times}$ are all distinct, then  the
equivalence class of $a$ has $m(q-1)$ elements. If they are not all
distinct then  $\sigma^i(a) = ka$ for some $i$, $i\not=0$, and some
$k \in F^{\times}$ (\cite[Lemma 23]{BP}). If $\sigma^i(a) = ka$
($i\not=0$) then $k$ is an $m$th root of unity, $k\not=1$. This
happens if and only if $m$ divides $q-1$.
\\ (i) If $m$ does not
divide $q-1$ then from $q^m-q$ elements in $K \setminus F$ we get
$(q^m - q)/(m(q-1)) $ equivalence classes.
\\ (ii) If $m$ divides $q-1$ then
$F$ contains all primitive $m$th roots of unity and so
$K = F(d)$ where $d$ is a root of an irreducible polynomial $t^m - c \in F[t]$.
 By \cite[Lemma 23]{BP},
the only elements $a \in K\setminus F$ with $\sigma^i(a) = ka$ are the
elements $d^j$, $1 \leq j \leq m-1$, and their $F$-scalar multiples.
Moreover, for each $d^j$, $\sigma^i(d^j) = \zeta^{ij}d^j$ and
$\zeta^{ij} \in F$, so there are only $q-1$ distinct elements in the
equivalence class of each $d^j$. Hence the $(q-1)(m-1)$ elements
$kd^j$ ($k \in F^{\times}$ and $j\in\{1,\dots, m-1\}$) form exactly
$m-1$ equivalence classes. Since these are all the elements in $K
\setminus F$ which are eigenvectors for the automorphisms $\sigma^i$,
the
 remaining $q^m-q - (q-1)(m-1)$ elements will form
\[\frac{q^m-q - (q-1)(m-1)}{m(q-1)}\]
equivalence classes. In total, we obtain
\[m-1 + \frac{q^m-q - (q-1)(m-1)}{m(q-1)}\]
equivalence classes.
\end{proof}

\begin{example}
Let $F = \mathbb{F}_3$ and $K = \mathbb{F}_9$, i.e.
\[K = F[x]/(x^2 - 2) = \{0,1,2,x,2x, x+1, x+2, 2x+1, 2x+2\}.\]
There are two non-isomorphic semifields which are nonassociative quaternion algebras with nucleus $K\cdot 1$, given by $A_1 = (K/F, \sigma, x)$ and
$A_{2}= (K/F, \sigma, x+1)$.
Now ${\rm Aut}_F(A_{1})\cong \mathbb{Z}/4 \mathbb{Z}$ whereas
${\rm Aut}_F(A_2)$ has order 8 and is isomorphic to the group  of quaternion units,
the smallest dicyclic group ${\rm Dic}_2$, by Theorem \ref{thm:semidirect}.

By Corollary \ref{numb},  these are  the only  non-isomorphic semifields
 of order $81$ of the type $(K/F, \sigma, a)$.
\end{example}


\end{document}